\documentclass[12pt]{amsart}
\usepackage[margin=1.35in]{geometry}
\usepackage{amscd,amsmath,amsxtra,amsthm,amssymb,stmaryrd,xr,mathrsfs,mathtools,enumerate,commath, comment, enumitem, graphicx}
\usepackage{stmaryrd}
\usepackage{multirow}
\usepackage{xcolor}
\usepackage{commath}
\usepackage{comment}
\usepackage{tikz-cd}
\usepackage{tkz-graph}

\usepackage{longtable} 
\usepackage{pdflscape} 
\usepackage{booktabs}
\usepackage{hyperref}
\definecolor{vegasgold}{rgb}{0.77, 0.7, 0.35}
\definecolor{darkgoldenrod}{rgb}{0.72, 0.53, 0.04}
\definecolor{gold(metallic)}{rgb}{0.83, 0.69, 0.22}
\hypersetup{
 colorlinks=true,
 linkcolor=darkgoldenrod,
 filecolor=brown,      
 urlcolor=gold(metallic),
 citecolor=darkgoldenrod,
 }
\newtheorem{lthm}{Theorem}

\usepackage[all,cmtip]{xy}

\DeclareFontFamily{U}{wncy}{}
\DeclareFontShape{U}{wncy}{m}{n}{<->wncyr10}{}
\DeclareSymbolFont{mcy}{U}{wncy}{m}{n}
\DeclareMathSymbol{\Sh}{\mathord}{mcy}{"58}
\usepackage[T2A,T1]{fontenc}
\usepackage[OT2,T1]{fontenc}

\usepackage{tikz}
\usetikzlibrary{shapes.geometric}
\usetikzlibrary{decorations.markings}
\tikzset{every loop/.style={min distance=10mm,looseness=10}}
\tikzstyle{vertex}=[auto=left,circle,minimum size=1pt,inner sep=0pt]

\newtheorem{theorem}{Theorem}[section]
\newtheorem{lemma}[theorem]{Lemma}
\newtheorem{ass}[theorem]{Assumption}
\newtheorem*{theorem*}{Theorem}
\newtheorem*{ass*}{Assumption}
\newtheorem{definition}[theorem]{Definition}
\newtheorem{corollary}[theorem]{Corollary}
\newtheorem{remark}[theorem]{Remark}

\newtheorem{proposition}[theorem]{Proposition}

\newcommand{\cG}{\mathcal{G}}

\newcommand{\Z}{\mathbb{Z}}

\newcommand{\Q}{\mathbb{Q}}

\newcommand{\cO}{\mathcal{O}}
\newcommand{\cS}{\mathcal{S}}

\newcommand{\cyc}{\mathrm{cyc}}

\newcommand{\op}[1]{\operatorname{#1}}

 \newcommand{\widebar}[1]{\mkern 2.5mu\overline{\mkern-2.5mu#1\mkern-2.5mu}\mkern 2.5mu}

\numberwithin{equation}{section}

\begin{document}

\title[Iwasawa theory of Cayley graphs]{On the Iwasawa theory of Cayley graphs}

\author[S.~Ghosh]{Sohan Ghosh}
\address[Ghosh]{Harish Chandra Research Institute, A CI of Homi Bhabha National Institute,  Chhatnag Road, Jhunsi, Prayagraj (Allahabad) 211 019 India}
\email{ghoshsohan4@gmail.com}

\author[A.~Ray]{Anwesh Ray}
\address[Ray]{Chennai Mathematical Institute, H1, SIPCOT IT Park, Kelambakkam, Siruseri, Tamil Nadu 603103, India}
\email{anwesh@cmi.ac.in}


\maketitle

\begin{abstract}
This paper explores Iwasawa theory from a graph theoretic perspective, focusing on the algebraic and combinatorial properties of Cayley graphs. Using representation theory, we analyze Iwasawa-theoretic invariants within $\mathbb{Z}_\ell$-towers of Cayley graphs, revealing connections between graph theory, number theory, and group theory. Key results include the factorization of associated Iwasawa polynomials and the decomposition of $\mu$- and $\lambda$-invariants. Additionally, we apply these insights to complete graphs, establishing conditions under which these invariants vanish.
\end{abstract}

\section{Introduction}
\subsection{Background and motivation}
\par Fix a prime number $\ell$ throughout. Let's consider a number field $K$, and let $\mathbb{Z}_\ell$ denote the ring of $\ell$-adic integers. An infinite abelian extension $K_\infty/K$ is said to be a $\Z_\ell$-extension if $\op{Gal}(K_{\infty}/K)$ is isomorphic to $\Z_\ell$ as a topological group. For each integer $n\geq 0$, set $K_n/K$ to be the extension contained in $K_\infty$ of degree $\ell^n$. Classical Iwasawa theory studies $\Z_\ell$-extensions of number fields and the asymptotic behavior of certain arithmetic invariants for the fields $K_n$. In the late 1950s, Iwasawa \cite{IwasawaMain} investigated the growth of class groups over these $\Z_\ell$-extensions of $K$, which laid the foundation for Iwasawa theory. Denote by $\op{Cl}(K_n)$ the class group of $K_n$ and $h_{K_n}:=\# \op{Cl}(K_n)$ its class number.

\begin{theorem*}[Iwasawa]
Let $K$ be a number field, $K_\infty$ be a $\Z_\ell$-extension of $K$ and for $n\in \Z_{\geq 0}$ denote by $\ell^{e_n}$ the exact power of $\ell$ that divides $h_{K_n}$. Then there exist invariants $\lambda,\mu\in \Z_{\geq 0}$ and $\nu\in \Z$,  depending on $\ell$ and independent of $n$, such that $e_n=\lambda n+\mu \ell^n+\nu$ for $n\gg 0$.  

\end{theorem*}

A natural example of a $\Z_\ell$-extension of $K$ is the cyclotomic $\Z_\ell$-extension, denoted by $K_\cyc$. For $K_\infty=K_\cyc$,  we simply denote the Iwasawa invariants by $\mu_\ell(K)${\color{blue},} $\lambda_\ell(K)${\color{blue},} and $\nu_\ell(K)$. Iwasawa famously conjectured that   $\mu_\ell(K)=0$ for all number fields $K$. This conjecture has been proved for abelian number fields by Ferrero and Washington \cite{FerreroWash}. In case of general number fields, this is still an open problem.

\par The Iwasawa theory of graphs was introduced by Vallières \cite{Vallieres:2021} and Gonet \cite{Gonet:2021a, Gonet:2022}. They explored $\Z_\ell$-towers of \emph{multigraphs} and demonstrated parallels to Iwasawa's theorem regarding the asymptotic variation of graph complexities along these towers, cf. Theorem \ref{GonetVallieresthm} for further details. The $\Z_\ell$-towers of multigraphs exhibit associated $\mu$, $\lambda$, and $\nu$-invariants, as well as a graph-theoretic interpretation of the \emph{Iwasawa polynomial}. Interest in Iwasawa theory of graphs has recently gained momentum, cf. for instance, the following works \cite{mcgownvallieresII, mcgownvallieresIII, DuBose/Vallieres:2022, Kleine/Muller:2022,Ray/Vallieres:2022, DLRV, LeiMuller}.

\par In this article, we take a closer look at Cayley graphs associated to finite abelian groups. There exists a captivating relationship between the algebraic properties of the Artin-Ihara L-functions associated with these graphs and the representation theory of their underlying groups. We analyze Iwasawa-theoretic invariants associated to Cayley graphs and show that there is an analogous relationship with representation theory. This leads to intriguing connections between graph theory, number theory, and group theory.

\subsection{Main results}
\par Let us describe our main results in greater detail. Let $G$ be a finite abelian group $G$ and $S\subseteq G$ be a subset such that
\begin{itemize}
    \item $S$ generates $G$,
    \item $S=S^{-1}$,
    \item $1\notin S$.
\end{itemize}
Set $\widehat{G}$ to denote the group of characters $\op{Hom}\left(G, \mathbb{C}^\times\right)$. The Cayley graph $\mathrm{Cay}(G, S)$ of the group $G$ with respect to the generating set $S$ is a graph where:
   \begin{itemize}
       \item The vertex set $V$ consists of elements of $G$, i.e., $V = \{v_g \mid g \in G\}$.
       \item There is an edge from vertex $g_1$ to vertex $g_2$ (denoted as $e(g_1, g_2)$) if and only if $g_1 g_2^{-1} \in S$.
   \end{itemize}
The eigenvalues of the adjacency matrix of a Cayley graph are closely related to the irreducible representations of the group $G$. We mention here that these Cayley graphs are related to a sightly different construction, namely Cayley--Serre graphs. These graphs are obtained as voltage assignments on bouquet graphs. For further details, we refer to \cite[p.445]{Vallieres:2021}. When $G$ is abelian, we show that Iwasawa polynomials associated to a Cayley graph can be factored in a natural way, where each character $\psi\in \widehat{G}$ gives rise to a factor. Let $\beta: S\rightarrow \Z_\ell$ be a function such that the following conditions are satisfied
\begin{enumerate}
\item the image of $\beta$ generates $\Z_{\ell}$, 
\item $\beta(s^{-1})=-\beta(s)$ and $\beta(1_G)=0$,
\item the image of $\beta$ lies in $\Z$,
\item there is a tuple $(h_1, \dots, h_m)\in S^m$ such that $h_1 h_2 \dots h_m\in S$ and
\begin{equation}
\beta(h_1 h_2 \dots h_m)\not \equiv \sum_{i=1}^m \beta(h_i)\pmod{\ell}.
\end{equation}
\end{enumerate}
Then associated to such a function is a $\Z_\ell$-tower of connected graphs
\begin{equation}\label{Z_l tower}X  \leftarrow X_1 \leftarrow X_2 \leftarrow \ldots \leftarrow X_k \leftarrow \ldots\end{equation}over $X=\rm{Cay}(G,S)$, this is made precise in subsection \ref{section 3.1}. Let $f_X(T)$ be the associated Iwasawa polynomial (see Definition \ref{definition of the Iwasawa polynomial}).

\begin{lthm}[Theorem \ref{determinant calculation}]
    There are explicit polynomials $P_\psi(T)$ (see \eqref{def of Pchi}) such that up to multiplication by a unit in $\Z_\ell\llbracket T\rrbracket$, the Iwasawa polynomial $f_X(T)$ is equal to the product $\prod_{\psi\in \widehat{G}} P_\psi(T)$.
\end{lthm}

Moreover, the $\mu$- and $\lambda$-invariants associated to the $\Z_\ell$-tower \eqref{Z_l tower} decompose into a sum of $\mu$- and $\lambda$-invariants associated to $P_\psi(T)$. In greater detail, $P_\psi(T)$ is a polynomial with coefficients in a valuation ring $\cO$ with uniformizer $\varpi$. Write $P_\psi(T)= \varpi^{\mu_\psi} Q_\psi(T) u_\psi(T)$, where $Q_\psi(T)$ is a distinguished polynomial and $u_\psi(T)$ is a unit in $\Z_\ell\llbracket T\rrbracket$. Set $\lambda_\psi$ to denote the degree of $Q_\psi(T)$. Set $e\in \Z_{\geq 1}$ to denote the \emph{ramification index}, defined by the relationship $(\ell)=(\varpi^e)$.

\begin{lthm}[Theorem \ref{sum of Iwasawa invariants theorem}]
    Let $\kappa_\ell(X_n)$ denote the $\ell$-primary part of the complexity of $X_n$. Then, for large enough values of $n$, we have that $\kappa_\ell(X_n)=\ell^{e_n}$, where 
    \[e_n=\mu \ell^n +n \lambda+ \nu,\] where 
    \[\mu=\frac{1}{e}\sum_\psi \mu_\psi\text{ and }\lambda=\sum_\psi \lambda_\psi-1,\] for the ramification constant $e\in \Z_{\geq 0}$ defined above.
\end{lthm}

Next, we show that the Iwasawa invariants associated to the factor polynomials $P_\psi(T)$ can be suitably calculated. Under some very explicit combinatorial criteria, it is shown that $\mu_\psi$ and $\lambda_\psi$ vanish. More specifically, we refer to Lemma \ref{basic lemma on mu and lambda vanishing}, Lemma \ref{trivial chi case lemma} and Proposition \ref{mu vanishing prop}.

\par Finally, our findings on Cayley graphs are utilized to examine the Iwasawa theory of complete graphs. Consider a positive integer $n$ and let $K_n$ represent the complete graph with no self-loops on $n$ vertices. Denote by $C_n$ the cyclic group of order $n$, and define $C_n' = C_n\setminus \{1\}$. The complete graph $K_n$ is then the Cayley graph associated to the pair $(C_n, C_n')$. We factor the Iwasawa polynomial, and obtain the following result on the vanishing of the $\mu$-invariant. 

\begin{lthm}[Theorem \ref{last thm}]
    For $\beta: C_n'\rightarrow \Z_\ell$, and $n\in \Z_{\geq 1}$ let $\mu_n$ (resp. $\lambda_n$) be the $\mu$-invariant (resp. $\lambda$-invariant) associated to the tower over $K_n$. Suppose that $\ell\nmid n$ and $\sum_{s\in C_n'} \bar{\beta}(s)^2\neq 0$. Then, $\mu_n=0$ and $\lambda_n=1$.
\end{lthm}
In our final section, we illustrate our results through two concrete examples. In fact, we are able to draw the graphs in towers, which makes them come to life, cf. \eqref{picture} and \eqref{picture2}. We note that an undirected Cayley graph $X=\op{Cay}(G, S)$ can be interpreted as a Galois cover of a bouquet. The formula $f_X(T)=\prod_{\psi\in \widehat{G}} P_\psi(T)$ in Theorem \ref{determinant calculation} has parallels with \cite[(4.5) on p. 20]{Ray/Vallieres:2022}. The latter formula, due to the second author and Daniel Vallieres, emerges in the context of proving an analogue of Kida's formula for cyclic groups within the graph-theoretic framework. This coincidence indicates that broader generalizations of our work are worth exploring. Such generalizations may encompass more general types of graphs, higher-dimensional analogues, or new classes of groups $G$.

\subsection{Organization}
The article is structured into five sections. Section \ref{s 2} covers foundational concepts and establishes relevant notation, including the formal definition of a \emph{multigraph}, discussion on the Galois theory of graph covers, and introduction of Artin-Ihara L-functions associated with graphs. Section \ref{s 3} focuses on establishing the Iwasawa theory of graphs, which involves a specific combinatorial framework for parameterizing connected $\Z_\ell$-towers over a graph. This section details the Iwasawa polynomial and its connection to the asymptotic complexity growth of graphs within $\Z_\ell$-towers, encapsulated by the Iwasawa $\mu$- and $\lambda$-invariants. In section \ref{s 4}, the main results of the article are proven. These results are illustrated through an example presented in section \ref{s 5}.

\subsection*{Data availability} The manuscript has no associated data.

\subsection*{Conflict of interest} There is no conflict of interest that the authors wish to report.

\subsection*{Acknowledgement} We would like to thank the referees for the careful review of the manuscript and the comments and corrections.

\section{Preliminaries}\label{s 2}
\par In this section, we recall some preliminary notions and set up relevant notation. The notation we use is consistent with \cite{Vallieres:2021, mcgownvallieresII, mcgownvallieresIII, Ray/Vallieres:2022, DLRV}.  In order to be consistent with previous work, we discuss generalities for multigraphs. The Cayley graphs we consider in this article are undirected graphs, with no self loops.

\subsection{Galois theory of covers}

\par We fix a prime number $\ell$ and set $\Z_\ell$ to denote the $\ell$-adic integers. Let $X$ be a finite multigraph; recall that this means that $X$ can be described as a quadruple $(V_X, E_X^+, i, \iota)$, where 
\begin{itemize}
    \item $V_X=\{v_1, \dots, v_{g_X}\}$ is a finite set of \emph{vertices},
    \item $E_X^+$ is a collection of \emph{edges} between vertices.
    \item  The set of edges is equipped with an \emph{incidence function} \[i: E_X^+\rightarrow V_X\times V_X.\] Here, the interpretation is that an edge $e$ starts at $v_i$ and ends at $v_j$ if $i(e)=(v_i, v_j)$,
    \item $\iota: E_X^+\rightarrow E_X^+$ is the \emph{inversion map}.
\end{itemize}
In addition, it is required that the following compatibility relations hold
\begin{enumerate}
    \item $\iota^2$ is the identity on $E_X^+$, 
    \item $\iota(e)\neq e$ for all $e\in E_X^+$,
    \item $i(\iota(e))=\tau(i(e))$,
\end{enumerate}
where $\tau: V_X\times V_X\rightarrow V_X\times V_X$ is the map defined by $\tau(e_1, e_2):=(e_2, e_1)$.
We write $e\sim e'$ if $e'=\iota(e)$, and denote by $E_X$ the set of equivalence classes for this relation. We think of $E_X^+$ as edges with orientation (or directed edges) and $E_X$ as undirected edges. The map $\pi: E_X^+\rightarrow E_X$ is the map that assigns to $e$ its equivalence class. When using the word "edge" we shall mean an element in $E_X^+$. We shall henceforth simply use the term "\emph{graph}" to refer to a multigraph. An edge from $v_i$ to $v_i$ is referred to as a \emph{loop}. We also define the incidence matrix $A_X=(a_{i,j})$ of the graph $X$, where $a_{i,j}$ is the number of edges from $v_i$ to $v_j$. We define source and target maps $o, t: E_X^+\rightarrow V_X$ to be the compositions of $i$ with the projections to the first and second factor of $V_X\times V_X$ respectively. Observe that the source map $o$ assigns to each directed edge its starting vertex, while the target map $t$ assigns to each edge its ending vertex. In particular, $e$ is a loop precisely if $o(e)=t(e)$. Note that $(V_X, E_X)$ is an undirected graph. For $v\in V_X$, let $E_{X,v}^+:=\{e\in E_X^+\mid o(e)=v\}$, i.e., the set of directed edges emanating from $v$. The \emph{degree} of $v$ is defined as the number of edges emanating from $v$, i.e., $\op{deg}(v):=\# E_{X,v}^+$. The betti numbers of $X$ are defined as follows 
\[b_i(X):=\op{rank}_{\Z} H_i(X,\Z).\] The \emph{Euler characteristic} is defined as follows $\chi(X):=b_0(X)-b_1(X)$. When $X$ is connected, $b_0(X)=1$ and $b_1(X)=\# E_X-\# V_X+1$. We have that $\chi(X)=\# V_X-\# E_X$.
 \begin{ass}\label{no vertices with val 1}
     It will be assumed throughout that all our multigraphs are connected with no vertices having degree equal to $1$. Moreover, we assume that $\chi(X)\neq 0$, i.e., the graph is not a cycle graph.
 \end{ass}
\par The divisor group $\op{Div}(X)$ consists of formal sums of the form $D=\sum_{v\in V_X} n_v v$, where $n_v\in \Z$ for all $v$. It is the free abelian group on the vertices $V_X$ of $X$. The degree of $D$ is the sum $\op{deg}(D):=\sum_v n_v\in \Z$, which defines a homomorphism
\[\op{deg}:\op{Div}(X)\rightarrow \Z,\] the kernel of which is denoted $\op{Div}^0(X)$. Let $\mathcal{M}(X)$ be the abelian group of $\Z$-valued functions on $V_X$. We note that $\mathcal{M}(X)$ can be freely generated by the characteristic functions $\chi_v$ defined by
\[\chi_v(v'):=\begin{cases}
&1 \text{ if }v'=v;\\
&0 \text{ otherwise.}\\
\end{cases}\]
Set $\op{div}(\chi_v):=\sum_{w\in V_X} \rho_{w}(v)w$, where
\[\rho_w(v):=\begin{cases}
\op{val}_X(v)-2 \cdot \text{ number of loops at }v & \text{ if }w=v;\\
-\text{ number of edges from }w\text{ to }v &\text{ if }w\neq v.
\end{cases}\]
With this notation in hand, we extend $\op{div}(\chi_v)$ to a map ${\rm div}:\mathcal{M}(X) \to {\rm Div}(X)$ as follows. For $f \in \mathcal{M}(X)$, one has that
$${\rm div}(f) = -\sum_{v}m_{v}(f) \cdot v,$$
where
$${m_{v}(f)} = \sum_{e \in E_{X,v}^+}\left(f(t(e)) - f(o(e))\right).$$
Note that $\op{div}(f)$ has degree $0$. We let $\op{Pr}(X)$ be the image of $\op{div}$, and set $\op{Pic}^0(X):=\op{Div}^0(X)/\op{Pr}(X)$. Set $\kappa_X$ to be the cardinality of $\op{Pic}^0(X)$. This quantity is analogous to the notion of the \emph{class number of a number ring} and referred to in the literature as the \emph{complexity} of $X$. For a more comprehensive account, see \cite{divisorsandsandpiles}.

\par We come to the notion of a Galois cover of graphs. First, we introduce the notion of a \emph{morphism} $f:Y\rightarrow X$ between graphs. This consists of a pair $(f_V, f_E)$, where $f_V: V_Y\rightarrow V_X$ and $f_E: E_Y^+\rightarrow E_X^+$ are functions satisfying the following compatibility relations:

\begin{enumerate}[label=(\alph*)]

\item $f_{V}(o(e)) = o(f_E(e))$,

\item $f_{V}(t(e)) = t(f_E(e))$,

\item $\iota\left(f_{E}(e)\right) = f_{E}\left(\iota(e)\right)$.

\end{enumerate}

We use $f$ to denote $f_V$ or $f_E$, depending on the context.\begin{definition} \label{cover}

Let $X$ and $Y$ be two graphs and $f:Y\rightarrow X$ be a graph morphism. If $f$ satisfies the following conditions:

\begin{enumerate}[label=(\alph*)]

\item $f:V_{Y} \rightarrow V_{X}$ is surjective,

\item for all $w \in V_{Y}$, the restriction $f|_{E_{Y,w}^+}$ induces a bijection

$$f|_{E_{Y,w}^+}:E_{Y,w}^+ \stackrel{\approx}{\rightarrow}{E_{X,f(w)}^+}, $$
\end{enumerate}then $f$ is said to be a \emph{cover}. The cover $f:Y \rightarrow X$ is called \emph{Galois} if the following two conditions are satisfied.

\begin{enumerate}

\item The graphs $X$ and $Y$ are connected.

\item The group ${\rm Aut}_{f}(Y/X):=\{ \sigma \in {\rm Aut}(Y) \, : \, f \circ \sigma = f\}$ acts transitively on the fiber $f^{-1}(v)$ for all $v \in V_{X}$.

\end{enumerate}

\end{definition}

We denote a Galois cover $f:Y\rightarrow X$ also by $Y/X$ and suppress the role of the covering map. Moreover, we set $\op{Gal}(Y/X):=\op{Aut}_f(Y/X)$.

\subsection{Artin--Ihara L-functions}
\par Before introducing the Iwasawa theory of graphs, we discuss the role of Artin-Ihara L-functions. These are essentially graph-theoretic analogs of Artin L-functions. The standard reference for the content of this subsection is \cite{Terras:2011}.

\par Let $X$ be a graph and $a_1, \dots, a_k\in E_X^+$ such that for $i<k$, one has that $t(a_i)=o(a_{i+1})$. Then, the sequence $a_1, \dots, a_k$ gives rise to a \emph{walk} $w=a_1a_2\dots a_k$. Here, $k$ is the \emph{length} of $w$ and is denoted by $l(w)$. The walk $w$ is said to have a \emph{backtrack} if $\iota(a_i)=a_{i+1}$ for some $i<k$. It is said to have a \emph{tail} if $a_k=\iota(a_1)$. A \emph{cycle} is a walk such that $o(a_1)=t(a_k)$. A cycle $\mathfrak{c}$ is said to be a prime if it has no backtrack or tail, and there is no cycle $u$ and integer $f>1$ for which $\mathfrak{c}\neq u^f$. In other words, $\mathfrak{c}$ is a prime if one can go around it only once. Let $Y/X$ be a Galois cover with abelian Galois group $G:=\op{Gal}(Y/X)$. Set $\widehat{G}:=\op{Hom}(G, \mathbb{C}^{\times})$ to be the group of characters of $G$. Given a character $\psi \in \widehat{G}$, the Artin-Ihara L-function is defined as follows:

\[L_{Y/X}(u, \psi):=\prod_{\mathfrak{c}}\left(1 - \psi\left(\frac{Y/X}{\mathfrak{c}} \right)u^{l(\mathfrak{c})} \right)^{-1}.\]

In the above product, $\mathfrak{c}$ runs over all primes of $X$ and $\left(\frac{Y/X}{\mathfrak{c}} \right)\in G$ refers to the Frobenius automorphism at $\mathfrak{c}$ (cf. \cite[Definition 16.1]{Terras:2011}). When $Y\rightarrow X$ is the identity $X\xrightarrow{\op{Id}} X$, we recover the \emph{Ihara zeta function} $\zeta_X(u):=L_{X/X}(u, 1)$.

\par Given a connected graph $X$, let $\chi(X)$ denote its \emph{Euler characteristic}, which is defined as follows:

\[\chi(X)=|V_X|-|E_X|.\]

It follows from the Assumption \ref{no vertices with val 1} that $\chi(X)< 0$.

Let $Y/X$ represent a covering of graphs with the property that it is abelian, having an automorphism group denoted as $G$. Suppose, for each $i$ from 1 to $g_X$, $w_i$ is a specific vertex chosen from the fiber of $v_i$. Considering $\sigma$ as an element of $G$, we define the matrix $A(\sigma)$ as a $g_X \times g_X$ matrix, with its entries denoted by $a_{i,j}(\sigma)$. This is determined as follows:

\begin{equation*}
a_{i,j}(\sigma) =
\begin{cases}
\text{Twice the number of loops at the vertex }w_{i}, &\text{ if } i=j \text{ and } \sigma = 1;\\
\text{The number of edges connecting $w_i$ to $w_{j}^{\sigma}$}, &\text{ otherwise}.
\end{cases}
\end{equation*}

For $\psi \in \widehat{G}$ (the character group of $G$), we define $A_\psi$ as the twisted sum of the $A(\sigma)$ matrices

\[A_{\psi} =A_{\psi, X}:= \sum_{\sigma \in G} \psi(\sigma) \cdot A(\sigma).\]

Let $D$ be the matrix $(d_{i,j})$ with \[d_{i, j}:=\begin{cases}
    0 &\text{ if }i\neq j;\\
    \op{deg}(v_i) &\text{ if }i=j.
\end{cases}\]

If $Y/X$ serves as an abelian covering of multigraphs and $\psi$ represents a character of $G = \op{Gal}(Y/X)$, then according to the three-term determinant formula from \cite[Theorem 18.15]{Terras:2011} applied to the Artin-Ihara $L$-function, we have:

\begin{equation}\label{LY/X formula} L_{Y/X}(u,\psi)^{-1} = (1-u^{2})^{-\chi(X)} \cdot \text{det}(I - A_{\psi}u + (D-I)u^{2}).\end{equation}
We let $h_X(u, \psi):=\op{det}\left(I - A_{\psi}u + (D-I)u^{2}\right)$ and for ease of notation, set $h_X(u):=h_X(u, 1)$. The result below gives us an explicit relationship between the derivative of $h_X$ and the complexity of $X$.

\begin{theorem}\label{class number formula thm}
    For a graph $X$ satisfying Assumption \ref{no vertices with val 1}, one has that 
    \[h_X'(1)=-2\chi(X) \kappa_X.\]
\end{theorem}
\begin{proof}
    For a proof of the result, we refer to \cite{Northshield} or \cite[Theorem 2.11]{HammerMattmanSandsVallieres}.
\end{proof}

For abelian covers, the Artin-Ihara L-functions described satisfy a relation as a consequence of the Artin formalism, as the following result shows.

\begin{theorem}\label{artin formalism thm}
    Let $Y/X$ be an abelian Galois cover of graphs with $G=\op{Gal}(Y/X)$, then one has
    \[\zeta_Y(u)=\zeta_X(u)\times \prod_{\psi\in \widehat{G}, \psi\neq 1} L_{Y/X}(u, \psi),\] where $1\in \widehat{G}$ denotes the trivial character.
\end{theorem}
\begin{proof}
    For a proof of this result, see \cite[Corollary 18.11]{Terras:2011}.
\end{proof}
The Artin formalism described above gives us a formula for the complexity of $Y$. This formula is expressed in terms of the complexity of $X$ and a product of special values of the \emph{twisted polynomials} $h_X(u, \psi)$. 
\begin{corollary}\label{class number relation corollary}
     Let $X$ be a graph satisfying Assumption \ref{no vertices with val 1} and $Y/X$ be an abelian Galois cover of $X$ with $G:=\op{Gal}(Y/X)$. Then, the following relationship holds
     \[|G|\kappa_Y=\kappa_X \prod_{\psi\in \widehat{G}, \psi\neq 1} h_X(1, \psi).\]
\end{corollary}

\begin{proof}
    The result is an immediate consequence of Theorem \ref{class number formula thm} and Theorem \ref{artin formalism thm}, see \cite[p.440]{Vallieres:2021} for further details.
\end{proof}

\begin{remark}Some remarks are in order. 
\begin{enumerate}
    \item It follows from the above relation in particular that $h_X(1, \psi)\neq 0$ for all $\psi\in \widehat{G}$ such that $\psi\neq 1$.
    \item The special value of the Artin-Ihara L-function at $u=1$ has been studied by Hammer, Mattman, Sands and Vallieres, cf. \cite{HammerMattmanSandsVallieres}.
\end{enumerate}
\end{remark}

\section{Iwasawa theory of graphs}\label{s 3}
\par The Iwasawa theory of graphs is a branch of mathematics that combines topology, combinatorics, and the Galois theory of covers of graphs. In this section, we summarize the key ideas and set up relevant notion. For a comprehensive account, please see \cite{Vallieres:2021,Gonet:2021a, mcgownvallieresII, mcgownvallieresIII}.
\subsection{$\Z_\ell$-covers of graphs}\label{section 3.1}
Throughout this section, we choose a section \[\gamma: E_X\rightarrow E_X^+\] of $\pi$ and set $\mathcal{S}:=\gamma(E_X)$. This set $\mathcal{S}$ is referred to as an \emph{orientation} of $X$, the understanding here is that for any edge, the orientation prescribes a direction to this edge. \begin{definition}
    Let $\cG$ be a finite group, a \emph{voltage} assignment valued in $\cG$ is a function $\alpha: \mathcal{S}\rightarrow \cG$. We extend $\alpha$ to all of $E_X^+$ by setting $\alpha(\iota(e)):=\alpha(e)^{-1}$.
\end{definition} 

   Let $(\mathcal{S}, \alpha)$ be a pair consisting of an orientation $\mathcal{S}$ and a voltage assignment $\alpha: \mathcal{S}\rightarrow \cG$. Let $(\mathcal{S}', \alpha')$ be another pair. Then we define as equivalence $(\mathcal{S}, \alpha)\sim (\mathcal{S}', \alpha')$ if the extensions of $\alpha$ and $\alpha'$ to $E_X^+$ coincide. Given a pair $(\mathcal{S}, \alpha)$ and another orientation $\mathcal{S}'$, then it is easy to see that there is a unique voltage assignment $\alpha': \mathcal{S}'\rightarrow \cG$ such that $(\mathcal{S}, \alpha)\sim (\mathcal{S}', \alpha')$. Associated to the datum $(\cG, \cS, \alpha)$ is a graph $X(\cG, \cS, \alpha)$, which we describe as follows. The set of vertices of $X(\cG, \cS, \alpha)$ is $V_X\times \cG$. On the other hand, the set of edges is identified with the set $E_X^+\times \cG$, where $(e, \sigma)\in E_X^+\times \cG$ is the edge that connects $(o(e), \sigma)$ to $(t(e), \sigma\cdot \alpha(e))$. On the other hand, the inversion map is given by $\overline{(e,\sigma)} = \left(\iota(e),\sigma \cdot \alpha(e)\right)$. It is clear that if $(\mathcal{S}, \alpha)\sim (\mathcal{S}', \alpha')$, then $X(\cG, \cS, \alpha)=X(\cG, \cS', \alpha')$.
\par This operation is functorial and results in Galois covers of $X$. Consider a voltage assignment $\alpha: \cS\rightarrow \cG$ and a group homomorphism $f:\cG \to \cG_{1}$. Then, there is a natural morphism of multigraphs denoted \( f_{*}:X(\cG,\mathcal{S},\alpha) \to X(\cG_{1},\mathcal{S},f \circ \alpha) \). This morphism is defined on vertices and edges as follows
\[ f_{*}(v,\sigma) = (v,f(\sigma)) \text{ and } f_{*}(e,\sigma) = (e,f(\sigma)). \]
If both \( X(\cG,\mathcal{S},\alpha) \) and \( X(\cG_1,\mathcal{S},f \circ \alpha) \) are connected, and \( f \) is surjective, then \( f_{*} \) is a cover according to Definition \ref{cover}. Moreover, it is a Galois cover, with the group of covering transformations being isomorphic to $\text{ker}(f)$. In particular, if $f:\cG \to \{ 1\}$ represents the group morphism into the trivial group and both $X$ and $X(\cG,\mathcal{S},\alpha)$ are connected, then a Galois cover, denoted $f_{*}:X(\cG,\mathcal{S},\alpha) \to X$, is obtained, with the group of covering transformations isomorphic to $\cG$.
\par We now recall the notion of a $\Z_\ell$-tower of multigraphs.
\begin{definition} \label{tower}
Let $\ell$ be a rational prime, and let $X$ be a connected graph. A $\mathbb{Z}_{\ell}$-tower over $X$ consists of a series of covers of connected graphs

$$X = X_{0} \leftarrow X_{1} \leftarrow \ldots \leftarrow X_{n} \leftarrow \ldots$$such that for every positive integer $n$, the cover $X_n / X$ obtained by composing the covers is Galois with Galois group $\op{Gal}(X_n / X)$ isomorphic to $\mathbb{Z}/\ell^{n}\mathbb{Z}$.
\end{definition}
\par We call a function $\alpha: \cS\rightarrow \Z_\ell$ a $\Z_\ell$-valued voltage assignment. Let \( \alpha_{/n} \) denote the mod-\( \ell^n \) reduction of \( \alpha \). This yields a \( \mathbb{Z}_{\ell} \)-tower over \( X \):
\[ X  \leftarrow X(\mathbb{Z}/\ell\mathbb{Z},\mathcal{S},\alpha_{/1}) \leftarrow X(\mathbb{Z}/\ell^{2}\mathbb{Z},\mathcal{S},\alpha_{/2}) \leftarrow \ldots \leftarrow X(\mathbb{Z}/\ell^{k}\mathbb{Z},\mathcal{S},\alpha_{/k}) \leftarrow \ldots. \]
It turns out that all $\Z_\ell$-towers arise from voltage assignments, cf. \cite[section 2.3]{Ray/Vallieres:2022} for details. Setting $t := \#\mathcal{S}$, we choose an ordering $\mathcal{S}=\{s_1,\dots , s_t\}$ and represent $\alpha$ as a vector \( \alpha=(\alpha_1, \alpha_2,\dots, \alpha_t)\in \Z_\ell^t \), where \( \alpha_i \) denotes \( \alpha(s_i) \). To relate growth patterns in Picard groups to Iwasawa invariants, it is necessary to impose certain conditions on the multigraphs \( X(\Z/\ell^n \Z, \mathcal{S}, \alpha_{/n}) \).
\begin{ass}\label{main assumption}
We assume that the derived multigraphs $X(\Z/\ell^n \mathbb{Z}, S, \alpha_{/n})$ are connected for all $n\geq 0$.
\end{ass}
\par We describe an explicit condition that $X(\cG, \mathcal{S}, \alpha)$ is connected. If $w = a_{1} a_{2}  \ldots  a_{n}$ is a walk in $X$, then we define 
$$\alpha(w) = \alpha(e_{1}) \cdot \ldots \cdot \alpha(e_{n}) \in \cG. $$ Recall that $\alpha: \mathcal{S}\rightarrow \cG$ satisfies the condition that $\alpha(\iota(e))=\alpha(e)^{-1}$. 
This implies that if $c_{1}$ and $c_{2}$ are homotopically equivalent, then $\alpha(c_{1}) = \alpha(c_{2})$. Choose a vertex $v_{0} \in V_{X}$, and let $\pi_1(X, v_0)$ be the fundamental group of $X$ with base-point $v_0$. We deduce that $\alpha$ induces a group homomorphism
\begin{equation} \label{group_mor}
\rho_{\alpha}:\pi_{1}(X,v_{0}) \rightarrow \cG, 
\end{equation}
defined by $\rho_{\alpha}([\gamma]) = \alpha(\gamma)$.
\begin{theorem}\label{rho is surjective}
    Assume that $X$ is a connected graph. Then, $X(\cG, \mathcal{S}, \alpha)$ is connected if and only if $\rho_\alpha$ is surjective.
\end{theorem}
\begin{proof}
    For a proof of this result, cf. \cite[Theorem 2.11]{Ray/Vallieres:2022}.
\end{proof}

\subsection{The Iwasawa polynomial}
\par Let $X$ be a connected graph. The matrix $D_X=(d_{i,j})$ is the $g_X\times g_X$ matrix for which 
\[d_{i, j}:=\begin{cases}
    \op{deg}(v_i) &\text{ if }i=j;\\
     0 & \text{ if }i\neq j.\\
\end{cases}\]
 The difference matrix $Q_X:=D_X-A_X$ is referred to as the \emph{Laplacian matrix}.

\par The rings $\Z_\ell[x]$ and $\Z_\ell\llbracket T \rrbracket $ denote the polynomial ring and formal power series ring with coefficients in $\Z_\ell$, respectively. The ring $\Z_\ell[x;\Z_\ell]$ consists of expressions in the form $f(x)=\sum_{a} c_a x^a$, where $a\in \Z_\ell$ and the coefficients $c_a$ belong to $\Z_\ell$. Let $\alpha: \cS\rightarrow \Z_\ell$ be a function and extend $\alpha$ to $E_X^+$ such that $\alpha(\iota(e))=-\alpha(e)$. The matrix $M(x)$ associated to the pair $(X,\alpha)$ has entries in $\Z_\ell[x; \Z_\ell]$ and is defined by\begin{equation}\label{def of M} M(x)=M_{X,\alpha}(x):=D_X - \left(\sum_{\substack{e \in E_X^+ \\ {i}(e) = (v_{i},v_{j})}} x^{\alpha(e)}\right)_{(i,j)}.\end{equation} We also introduce the notation ${b\choose n}$ for $\frac{b(b-1)\dots (b-n+1)}{n!}$, and define $(1+T)^b$ as the formal power series $\sum_{n=0}^\infty {b\choose n} T^n$. 
\begin{definition}\label{definition of the Iwasawa polynomial}
    With respect to notation above, the \emph{Iwasawa polynomial} associated to the $\Z_\ell$-tower defined by $\alpha$ is defined as follows
    \[f_{X, \alpha}(T):=\op{det}M(1+T)\in \Z_\ell\llbracket T\rrbracket.\]
\end{definition}
As defined above, the $f_{X, \alpha}(T)$ is not in general a polynomial, however, becomes one upon multiplication by a suitably chosen unit in $\Z_\ell\llbracket T\rrbracket$. This unit can be taken to be a suitably large power of $(1+T)$. 
\par Let $\zeta_{\ell^{n}}$ be a primitive $\ell^n$-th root of unity. For the tower
\begin{equation}\label{alphatower}X  \leftarrow X(\mathbb{Z}/\ell\mathbb{Z},\cS,\alpha_{/1}) \leftarrow X(\mathbb{Z}/\ell^{2}\mathbb{Z},\cS,\alpha_{/2}) \leftarrow \ldots \leftarrow X(\mathbb{Z}/\ell^{n}\mathbb{Z},\cS,\alpha_{/n}) \leftarrow \ldots\end{equation}
related to $\alpha$, for any positive integer $n$, it follows from \cite[Corollary 5.6]{mcgownvallieresIII} that
\begin{equation} \label{spe_pow1}
f_{X,\alpha}(1- \zeta_{\ell^{n}}) = h_{X}(1,\psi_{n}), 
\end{equation}
where $\psi_{n}$ denotes the character of $\mathbb{Z}/\ell^{n}\mathbb{Z}$, defined by $\psi_{n}(\bar{1}) = \zeta_{\ell^{n}}$.
\begin{lemma}
    The Iwasawa polynomial $f_{X, \alpha}(T)$ is divisible by $T$. 
\end{lemma}

\begin{proof}
    We find that $f_{X,\alpha}(0)=\op{det}\left(M(1+0)\right)=\op{det}\left(Q_X\right)$. It is easy to see that $Q_X$ is a singular matrix, indeed, $u:=(1, 1, \dots, 1)^t$ is in its null-space. Therefore, $f_{X,\alpha}(0)=0$, in other words, $T$ divides $f_{X, \alpha}(T)$. 
\end{proof}

\par Note that $f_{X,\alpha}(T)$ is a Laurent series in $(1+T)$. In light of the above result, we write $f_{X,\alpha}(T)=(1+T)^{-m} T g_{X, \alpha}(T)$, where $m\in \Z_{\geq 0}$ is the smallest number such that $g_{X, \alpha}(T)$ above is a polynomial. Of significance are the Iwasawa invariants that are associated to the Iwasawa polynomial. We recall that a polynomial $g(T)\in \Z_\ell[T]$ is a \emph{distinguished polynomial} if it is a monic polynomial and all the non-leading coefficients of $g(T)$ are divisible by $\ell$.
It follows from the $\ell$-adic Weierstrass Preparation theorem, that there is a factorization of the Iwasawa polynomial \[g_{X,\alpha}(T)=\ell^{\mu} P(T) u(T),\] where $P(T)$ is a distinguished polynomial and $u(T)$ is a unit in $\Z_\ell\llbracket T\rrbracket $. Since $u(T)$ is a unit in $\Z_\ell\llbracket T\rrbracket$, the constant term $u(0)$ is a unit in $\Z_\ell$. In particular, we have that $\op{deg} g\geq 1$. The invariants $\mu_{\ell}(X,\alpha):=\mu$ and $\lambda_{\ell}(X,\alpha):=\op{deg} P(T)$ are the $\mu$- and $\lambda$-invariant associated to the tower \eqref{alphatower}.

\begin{theorem}[Gonet \cite{Gonet:2021a, Gonet:2022}, Vallieres \cite{Vallieres:2021}, McGown--Vallieres \cite{mcgownvallieresII, mcgownvallieresIII}]\label{GonetVallieresthm}
    Let $X$ be a graph and $\alpha$ a $\Z_\ell$-valued voltage assignment such that Assumptions \ref{no vertices with val 1} and \ref{main assumption} are satisfied. For $n\in \Z_{\geq 1}$, set $X_n:=X(\Z/\ell^n \Z, \cS, \alpha_{/n})$ and $\kappa_\ell(X_n)$ denote the $\ell$-primary part of the complexity of $X_n$. Let $\mu:=\mu_\ell(X, \alpha)$ and $\lambda:=\lambda_\ell(X, \alpha)$. Then, there exists $n_0>0$ and $\nu\in \Z$ such that for all $n\geq n_0$,
    \[\kappa_\ell(X_n)=\ell^{(\ell^n\mu+n \lambda+\nu)}.\]
\end{theorem}
\begin{proof}
    The result above is \cite[Theorem 6.1]{mcgownvallieresIII}.
\end{proof}
\begin{remark}
 Vallieres proved a restricted version of the above result for certain $\Z_\ell$-towers consisting of Cayley--Serre multigraphs \cite[Theorem 5.6]{Vallieres:2021}. This result was subsequently generalized by McGown and Vallieres. The result was proven independently via a different method by Gonet. 
\end{remark}
\section{Cayley graphs and Iwasawa theory}\label{s 4}
\subsection{A factorization of the Iwasawa polynomial}
\par Let $G$ be a finite abelian group and $S$ be a subset of $G$ such that\begin{itemize}
    \item $S$ generates $G$,
    \item $S=S^{-1}$,
    \item $1\notin S$.
\end{itemize} We shall set $r:=\# S$. Associated with the pair $(G, S)$ is a Cayley graph, denoted by $X:=\mathrm{Cay}(G, S)$. This Cayley graph is defined by taking the vertex set $V=\{v_g\mid g\in G\}$ as the elements of $G$ and connecting two vertices $g_1$ and $g_2$ by an edge $e(g_1, g_2)$ if $g_1 g_2^{-1}\in S$. Since $S=S^{-1}$, it follows that there is an edge $e$ starting at $g_1$ and ending at $g_2$, precisely if there is an edge $\iota(e)$ starting at $g_2$ and ending at $g_1$. Since $1_G\notin S$, it follows that $\mathrm{Cay}(G, S)$ has no loops. Throughout, we assume that the Assumption \ref{no vertices with val 1} is satisfied.

\begin{definition}\label{def of beta}
    We shall consider voltage assignments that arise from functions on $S$. Let $\ell$ be a prime number and $\beta: S\rightarrow \Z_\ell$ be a function such that: 
\begin{enumerate}
\item the image of $\beta$ generates $\Z_{\ell}$ (as a $\Z_\ell$-module), 
\item $\beta(s^{-1})=-\beta(s)$ and $\beta(1_G)=0$,
\item the image of $\beta$ lies in $\Z$,
\item there exists $m>0$ and a tuple $(h_1, \dots, h_m)\in S^m$ such that $h_1 h_2 \dots h_m\in S$ and
\begin{equation}\label{congruence equation}
\beta(h_1 h_2 \dots h_m)\not \equiv \sum_{i=1}^m \beta(h_i)\pmod{\ell}.
\end{equation}
\end{enumerate} 
We define a $\Z_\ell$-valued \emph{voltage assignment} $\alpha=\alpha_\beta: E_X^+\rightarrow \Z_\ell$ by $\alpha\left(e\right):=\beta(g_1 g_2^{-1})$ where $e$ is the edge joining $v_{g_1}$ to $v_{g_2}$.
\end{definition} 

\begin{remark}
    The condition (1) above is a consequence of (4). It follows from (2) that $\alpha(\iota(e))=\beta(g_2g_1^{-1})=-\beta(g_1 g_2^{-1})=-\alpha(e)$. Condition (3) simplifies some of our calculations, see for instance in the definition of $m_\beta$ in the formula for $P_\psi(T)$ in \eqref{def of Pchi}. Finally, (4) is used in the Proposition below to establish connectedness of the graphs in towers.
\end{remark}

\begin{proposition}
    With respect to notation above, the graphs $X(\Z/\ell^k, S, \alpha_{/k})$ are connected for all $k\geq 0$ (i.e. the Assumption \ref{main assumption} is satisfied).
\end{proposition}

\begin{proof}
    We deduce the result from Theorem \ref{rho is surjective} by showing that the homomorphism $\rho_\alpha^k:\pi_1(X, v_0)\rightarrow \Z/\ell^k \Z$ is surjective for all $k\geq 1$. Fix a tuple $(h_1, \dots, h_m)$ such that
    \[\beta(h_1 h_2 \dots h_m)\not \equiv \sum_{i=1}^m \beta(h_i)\pmod{\ell}.\] We take $v_0$ to be the vertex $v_{1_G}$. Consider the sequence of elements \[a_0:=1_G, a_1:=h_1, a_2:=h_1 h_2, \dots, a_m :=h_1 h_2\dots h_m.\] The sequence of elements gives rise to a loop $\gamma$ starting and ending at $v_0=v_{a_0}$ traversing the vertices $v_{a_1}, \dots, v_{a_m}$. For $i=1, \dots, m$, let $e_i$ be the edge joining $a_{i-1}$ to $a_{i}$ and $e_{m+1}$ be the edge joining $a_m$ back to $a_0$. We find that 
    \[\alpha(e_i)=\begin{cases}
        -\beta(h_i) & \text{ if } i<m+1;\\
        \beta(h_1\dots h_m)& \text{ if } i=m+1.
    \end{cases}\]
    Therefore, $\rho_\alpha^k(\gamma)=-\sum_{i=1}^m \beta(h_i)+\beta(h_1\dots h_m)$ and thus, $\rho_\alpha^k(\gamma)\in \left(\Z/\ell^k \Z\right)^\times$. Therefore, the map $\rho_\alpha^k$ is surjective for all $k\geq 1$.
\end{proof}

    The condition \eqref{congruence equation} can be specialized to a neater condition as follows.
    \begin{proposition}\label{reformation propn}
        Let $\beta: S\rightarrow \Z_\ell$ be a function satisfying the conditions (1)-(3) of Definition \ref{def of beta}. Assume moreover that there exists $h\in S$ with order $M>1$ such that the following conditions are satisfied
        \begin{enumerate}
            \item $(M, \ell)=1$,
            \item $\beta(h)\not\equiv 0\mod{\ell}$.
        \end{enumerate}
        Then, the congruence condition \eqref{congruence equation} is satisfied.
    \end{proposition}
    \begin{proof}
        Taking $h_1=h_2=\dots=h_{M-1}=h$, the congruence condition becomes 
    \[\beta(h^{M-1})\not \equiv (M-1)\beta(h)\pmod{\ell}.\]
    Note that $h^{M-1}=h^{-1}$ and hence is contained in $S$. The condition becomes $-\beta(h)\not \equiv (M-1)\beta(h)\pmod{\ell}$, i.e., $\ell\nmid M \beta(h)$.
    Since $(M, \ell)=1$, and $\beta(h)\not\equiv 0\pmod{\ell}$, the result follows.
    \end{proof}

\par We choose an ordering and write $G=\{g_1, \dots, g_n\}$ and set $v_i:=v_{g_i}$. For $g\in G$, set \[\delta_S(g):=\begin{cases}
1 &\text{ if }g\in S;\\ 
0 &\text{ if }g\notin S.
\end{cases}\]Recall that a voltage assignment $\alpha: E_X^+\rightarrow \Z_\ell$ gives rise to a $\Z_\ell$-tower over $X$ and that \[f_{X, \alpha}(T)=\op{det}\left(\op{M}_{X, \alpha}(1+T)\right)=\op{det}\left(d_{i,j}- \delta_S(g_i g_j^{-1})(1+T)^{\beta(g_i g_j^{-1})}\right)_{i,j}\] is the associated Iwasawa polynomial.

\par Let $\bar{\Q}$ (resp. $\bar{\Q}_\ell$) be a choice of algebraic closure of $\Q$ (resp $\Q_\ell$). Choose an embedding of $\bar{\Q}$ into $\bar{\Q}_\ell$ and thus view any algebraic number as an element in $\bar{\Q}_\ell$. Let $K$ be the field generated by $|G|$-th roots of unity in $\bar{\Q}_\ell$. Set $\cO$ to denote its valuation ring and $\varpi$ its uniformizer. Thus any character $\psi\in \widehat{G}$ takes values in $\cO^\times$. Take $k$ to denote the residue field $\cO/(\varpi)$ and set $\bar{\psi}:G\rightarrow k^\times$ to denote the mod-$\varpi$ reduction of $\psi$. Recall that $r:=\# S$. For $\psi \in \widehat{G}$, set 
    \begin{equation}\label{def of Pchi}P_\psi(T)=P_{X, \alpha, \psi}(T):=(1+T)^{m_\beta}r-\sum_{s\in S} (1+T)^{m_\beta-\beta(s)} \psi(s)\in \cO[T],\end{equation} where $m_\beta:=\op{max}\left\{\beta(s)\mid s\in S\right\}$. Given elements $f, g\in \cO\llbracket T\rrbracket$, write $f\sim g$ to mean that $f=ug$ for some unit $u$ of $\cO\llbracket T\rrbracket$. 
\begin{theorem}\label{determinant calculation}
    Let $G$ be a finite abelian group and $S$ a subset of $G$ such that \begin{itemize}
    \item $S$ generates $G$,
    \item $S=S^{-1}$,
    \item $1\notin S$.
\end{itemize} Let $X=\op{Cay}(G, S)$ be the associated Cayley graph. Let $\beta: S\rightarrow \Z_\ell$ be a function satisfying the conditions (1)--(4) of Definition \ref{def of beta} and \[\alpha=\alpha_\beta:E_X^+\rightarrow \Z_\ell\] be the associated $\Z_\ell$-valued voltage assignment. Then, we have that \[f_{X, \alpha}(T)\sim \prod_{\psi\in \widehat{G}} P_\psi(T).\]
\end{theorem}
\begin{proof}
    For $\psi\in \widehat{G}$, set $v_\psi=(\psi(g_1), \dots, \psi(g_n))$ observe that 
\[\left(M_{X,\alpha}(x) v_\psi\right)_i= r \psi(g_i)-\sum_{j; g_i g_j^{-1}\in S} x^{\beta(g_i g_j^{-1})} \psi(g_j)=r \psi(g_i)-\sum_{s\in S} x^{\beta(s)} \psi(g_i s^{-1}).\]
Therefore, 
\begin{equation}\label{matrix relation}M_{X,\alpha}(x) v_\psi= \left(r-\sum_{s\in S} x^{\beta(s)} \psi(s^{-1})\right)v_\psi.\end{equation}
Choose an enumeration of $\widehat{G}=\{\psi_1, \dots, \psi_j, \dots, \psi_n\}$, and let $\mathbf{F}:=\left(\psi_j(g_i)\right)_{i,j}$. It follows from the orthogonality relations that $\mathbf{F}$ is invertible with $\mathbf{F}^{-1}=\left(\frac{1}{n}\bar{\psi}_i(g_j)\right)_{i,j}$. The relation \eqref{matrix relation} implies that $v_{\psi_1}, \dots, v_{\psi_n}$ is a basis of eigenvectors and $\mathbf{F}$ is the change of basis matrix. Thus, \eqref{matrix relation} can be rephrased as follows
\[M_{X,\alpha}(x)= \mathbf{F}B_{X,\alpha}(x) \mathbf{F}^{-1},\] where 
\[B_{X, \alpha}(x)=\left(b_{i,j}(x)\right),\] where 
\[b_{i,j}(x)=\begin{cases}
    0 &\text{ if } i\neq j;\\
    \left(r-\sum_{s\in S} x^{\beta(s)} \psi_i(s^{-1})\right) &\text{ if }i=j.\\
\end{cases}\]
Therefore, 
\[\begin{split}f_{X, \alpha}(T)=& \op{det}\left(M_{X,\alpha}(1+T)\right) \\
= & \op{det}\left(D_{X, \alpha}(1+T)\right) \\
=& \prod_{\psi\in \widehat{G}} \left(r-\sum_{s\in S} (1+T)^{-\beta(s)} \psi(s)\right), \\
=& (1+T)^{-m_\beta n} \prod_{\psi\in \widehat{G}}P_\psi(T);\end{split}\]
since $(1+T)$ is a unit in $\Z_\ell\llbracket T\rrbracket$, the result follows.
\end{proof}

\begin{remark}Let $\beta, \beta': S\rightarrow \Z_\ell$ be functions satisfying the conditions of Definition \ref{def of beta}. Suppose that there is a constant $c\in \Z$ such that $\ell\nmid c$ and $\beta(s)=c \beta'(s)$ for all $s\in S$. Let $\alpha:=\alpha_\beta$ and $\alpha':=\alpha_{\beta'}$ be the associated $\Z_\ell$-valued functions. Then, it is easy to see that \[f_{X, \alpha'}(T)=f_{X,\alpha}\left((1+T)^c-1\right).\]
Note that $T\sim (1+T)^c-1$ and thus it follows from an application of the $\ell$-adic Weierstrass preparation theorem that 
\[\mu_\ell(X, \alpha)=\mu_\ell(X, \alpha')\text{ and }\lambda_\ell(X, \alpha)=\lambda_\ell(X, \alpha').\]
Thus, as far as the computation of Iwasawa invariants is concerned, we may as well assume that there are no common divisors of the values $\{\beta(s)\mid s\in S\}$ (other than $\pm 1$).
\end{remark}

\subsection{The Iwasawa invariants $\mu_\psi$ and $\lambda_\psi$ and their properties}
Write $P_\psi(T)=\varpi^{\mu_\psi} Q_\psi(T) u_\psi(T)$, where $Q_\psi(T)$ is a distinguished polynomial and $u_\psi(T)\in \cO\llbracket T\rrbracket^\times$. Set $\lambda_\psi$ to denote the $\lambda$-invariant of $P_\psi(T)$, defined as follows $\lambda_\psi:=\op{deg}\left(Q_\psi(T)\right)$. Let $e$ denote the ramification index, defined by $(\ell)=(\varpi^e)$ in $\cO$. Note that $P_{1}(0)=r-\sum_{s\in S} 1(s)=r-r=0$, and thus, $T$ divides $P_1(T)$. In particular, $\lambda_1\geq 1$.

\begin{theorem}\label{sum of Iwasawa invariants theorem}
    With respect to notation above, 
    \[\begin{split}
        &\mu_\ell(X, \alpha)=\frac{1}{e}\left(\sum_{\psi\in \widehat{G}} \mu_\psi\right);\\
        &\lambda_\ell(X, \alpha)=\left(\sum_{\psi\in \widehat{G}} \lambda_\psi\right)-1.
    \end{split}\]
    In particular, if $\mu_\ell(X, \alpha)=0$ if and only if $\mu_\psi=0$ for all $\psi\in \widehat{G}$. 
\end{theorem}
\begin{proof}
    By Theorem \ref{determinant calculation}, \[T g_{X, \alpha}(T)=f_{X, \alpha}(T)\sim \prod_{\psi\in \widehat{G}} P_\psi(T),\] and therefore 
    \[Tg_{X, \alpha}(T)\sim \varpi^{\sum_{\psi\in \widehat{G}} \mu_\psi} \prod_{\psi\in \widehat{G}} Q_\psi(T). \]
    Therefore, 
    \[\begin{split} & \sum_\psi \mu_\psi=e \mu_\ell(X,\alpha), \\ 
    & \sum_{\psi} \lambda_\psi=1+\lambda_\ell(X,\alpha),
    \end{split}\]
    from which the result follows.
\end{proof}

For each $j\in [0, 2 m_\beta]$, let \[S_j:=\beta^{-1}\{(m_\beta-j)\}=\{s\in S\mid \beta(s)=m_\beta-j\}.\] From \eqref{def of Pchi}, we deduce that 
\[P_\psi(T)=\sum_{j=0}^{2 m_\beta} a_{j, \psi} (1+T)^j,\]
where \begin{equation}\label{formula for the ajs}a_{j, \psi}:=\begin{cases}
     -\sum_{s\in S_j} \psi(s) &\text{ if }j\neq m_\beta;\\
     r-\sum_{s\in S_j} \psi(s) &\text{ if }j=m_\beta.
\end{cases}\end{equation}
Also note that $P_\psi(0)=r-\sum_{s\in S} \psi(s)$. Recall that $\bar{\psi}$ is the mod-$\varpi$ reduction of $\psi$. Set $\bar{r}$ to denote the mod-$\ell$ reduction of $r$.

\begin{lemma}\label{basic lemma on mu and lambda vanishing}
    Let $\psi$ be a nontrivial character of $G$, then the following are equivalent
    \begin{enumerate}
        \item $\mu_\psi=0$ and $\lambda_\psi=0$,
        \item $P_\psi(T)$ is a unit in $\cO\llbracket T\rrbracket$,
        \item $\sum_{s\in S} \bar{\psi}(s)\neq \bar{r}$.
    \end{enumerate}
\end{lemma}
\begin{proof}
    The equivalence of the conditions (1) and (2) is an easy consequence of the $\ell$-adic Weierstrass preparation theorem. The power series $P_\psi(T)$ is a unit if and only if its constant term $P_\psi(0)$ is a unit in $\cO$. Since $P_\psi(0)=r-\sum_{s\in S} \psi(s)$, it follows that (2) and (3) are equivalent.
\end{proof}

\begin{lemma}\label{trivial chi case lemma}
    With respect to notation above, $\lambda_1\geq 2$. Moreover, the following conditions are equivalent
    \begin{enumerate}
        \item $\mu_1=0$ and $\lambda_1=2$, 
        \item $P_1''(T)$ is a unit in $\cO\llbracket T\rrbracket$,
        \item $\sum_{s\in S} \bar{\beta}(s)^2\neq 0$.
    \end{enumerate}
\end{lemma}
\begin{proof}
    Note that $P_1(0)=0$ and hence $P_1(T)$ is divisible by $T$. Recall that from \eqref{def of Pchi}, we have that 
    \[P_1(T)=(1+T)^{m_\beta}r-\sum_{s\in S} (1+T)^{m_\beta-\beta(s)}\]
    and thus, 
    \[P_1'(0)=m_\beta r-\sum_{s\in S} (m_\beta-\beta(s))=\sum_{s \in S} \beta(s).\]
    Since $\beta(s^{-1})=-\beta(s)$ and $S=S^{-1}$, it follows that $P_1'(0)=0$. Hence, $T^2$ divides $P_1(T)$ and $\lambda_1\geq 2$.
    \par We find that 
    \[\begin{split}P_1''(0)=& m_\beta (m_\beta-1) r-\sum_{s\in S} (m_\beta-\beta(s))(m_\beta-\beta(s)-1)\\ =&  m_\beta (m_\beta-1) r-\sum_{s\in S} \left(m_\beta(m_\beta-1)-2\beta(s) m_\beta +\beta(s)^2+\beta(s)\right)\\
    = & (1-2m_\beta)\sum_{s\in S} \beta(s)+\sum_{s\in S}\beta(s)^2=\sum_{s\in S}\beta(s)^2.
    \end{split}\]
    It is an immediate consequence of the $\ell$-adic Weierstrass preparation theorem that (1) and (2) are equivalent. We find that $P_1''(T)$ is a unit in $\cO\llbracket T\rrbracket$ if and only if $P_1''(0)$ is not divisible by $\ell$. On the other hand, note that $P_1''(0)=\sum_{s\in S}\beta(s)^2$ and therefore, (2) is equivalent to (3).
\end{proof}

\begin{proposition}\label{mu vanishing prop}
    Let $\beta: S\rightarrow \Z_\ell$ be a function satisfying the conditions of Definition \ref{def of beta}. Moreover, assume that there exists $j\neq 0$ such that $S_j$ is a singleton. Then, the following assertions hold
    \begin{enumerate}
        \item $\mu_\psi=0$ for all $\psi\in \widehat{G}$,
        \item $\mu(X, \alpha)=0$. 
    \end{enumerate}
\end{proposition}
\begin{proof}
    It follows from Theorem \ref{sum of Iwasawa invariants theorem} that the assertions (1) and (2) are equivalent. Thus, it suffices to show that $\mu_\psi=0$ for all $\psi\in \widehat{G}$. From \eqref{formula for the ajs}, we find that 
    \[\widebar{a_{j, \psi}}=\bar{\psi}(s),\] where $S_j=\{s\}$. Since $\bar{\psi}(s)$ is a root of unity, it follows therefore that $\widebar{a_{j, \psi}}\neq 0$. This implies that the mod-$\varpi$ reduction of $P_\psi(T)$ is nontrivial, which in turn implies that $\mu_\psi=0$. 
\end{proof}

\subsection{Complete graphs}

\par In this section, we let $n$ be a positive integer and $K_n$ be the complete graph on $n$ vertices. Let $C_n$ be the cyclic group of order $n$ and and $C_n':=C_n\setminus \{1\}$. Then, we find that $K_n=\rm{Cay}(C_n, C_n')$, we apply the results of the previous subsection to study the $\mu$- and $\lambda$-invariants of $K_n$. Let $\beta: C_n'\rightarrow \Z_\ell$ be a function satisfying the conditions of Definition \ref{def of beta}. Let $\mu_n$ and $\lambda_n$ be the Iwasawa invariants associated to $(K_n, \alpha)$. Let $\mu_{n, \psi}$ and $\lambda_{n, \psi}$ be the Iwasawa invariants $\mu_\psi$ and $\lambda_\psi$ introduced in the previous section. 

\begin{lemma}\label{mu n chi and lambda n chi are 0}
Assume that $\ell\nmid n$. Then, for all nontrivial characters $\psi$, 
\[\mu_{n, \psi}=\lambda_{n, \psi}=0.\]
\end{lemma}
\begin{proof}
    Recall that Lemma \ref{basic lemma on mu and lambda vanishing} asserts that \[\mu_{n, \psi}=\lambda_{n, \psi}=0\] if 
    \[\bar{r}\neq \sum_{s\in C_n'} \bar{\psi}(s).\] Since $\psi$ is nontrivial, $\sum_{s\in C_n'} \bar{\psi}(s)=-1$ and $\bar{r}=\widebar{(n-1)}$. Since $\ell\nmid n$, the result follows. 
\end{proof}
\begin{theorem}\label{last thm}
   Suppose that $\ell\nmid n$ and $\sum_{s\in C_n'} \bar{\beta}(s)^2\neq 0$. Then, $\mu_n=0$ and $\lambda_n=1$.  
\end{theorem}
\begin{proof}
    It follows from Lemma \ref{mu n chi and lambda n chi are 0} that $\mu_{n, \psi}=\lambda_{n, \psi}=0$ for all nontrivial $\psi\in \widehat{G}$. On the other hand, $\mu_{n, 1}=0$ and $\lambda_{n, 1}=2$ by Lemma \ref{trivial chi case lemma}. Therefore, it follows from Theorem \ref{sum of Iwasawa invariants theorem} that $\mu_n=0$ and $\lambda_n=1$. 
\end{proof}
\section{Illustrative examples}\label{s 5}
In this section, we illustrate our results through two examples.

\subsection{Example 1}
\par Let $G=\Z/4\Z:=\{0,1,2,3\}$, $S=\{1, 2, 3\}$, $X:=\rm{Cay}(G, S)$ and $\ell:=3$. Note that the condition $S=S^{-1}$ is satisfied. Moreover, the Assumption \ref{no vertices with val 1} is satisfied and $\chi(X)=-2$. Now, define the function $\beta:S \rightarrow \Z_3$ as follows
    \[
    \beta(s)=\begin{cases}
          1, & \mathrm{if}\  s=1;\\
          0, & \mathrm{if}\  s=2;\\
          -1, & \mathrm{if}\  s=3;
       \end{cases}
    \]
and $\alpha: E_X^+\rightarrow \Z_3$ the associated function. We check that the conditions (1)--(4) of Definition \ref{def of beta} are all satisfied.
\begin{enumerate}
    \item The image of $\beta$ is $\{-1, 0, 1\}$ and clearly generates $\Z_3$.
    \item The condition $\beta(-s)=-\beta(s)$ is easy to verify. 
    \item By definition, the image of $\beta$ lies in $\Z$.
    \item We apply Proposition \ref{reformation propn} to verify the condition (4). Indeed setting $h:=1$, this element has order $m=4$. We note that $(m, \ell)=1$ and $\beta(h)=1$. Also, $h, h^2$ and $h^3$ are contained in $S$.
\end{enumerate}
The matrix $M(1+T)$ (see \eqref{def of M}) is given by 
    \[
      \begin{pmatrix}
        3 & -(1+T)^{-1} & -1 & -(1+T)\\
        -(1+T) & 3 &  -(1+T)^{-1}& -1\\
        -1 & -(1+T) & 3 & -(1+T)^{-1}\\
         -(1+T)^{-1}& -1  & -(1+T)& 3 
    \end{pmatrix}
    \]
Let $\psi\in \widehat{G}$ be the character which is defined by $\psi(n):=\op{exp}\left(\frac{2\pi \mathbf{i} 
 n}{4}\right)=\mathbf{i}^n$, where $\mathbf{i}$ is a squareroot of $-1$. Set $P_j(T):=P_{\psi^j}(T)$ for $j=0, \dots, 3$; note that $m_\beta=1$. Setting $x:=(1+T)$, we find that 
 \[\begin{split}P_j(T)=& -\mathbf{i}^{3j}x^{2}+ (3-\mathbf{i}^{2j})x-\mathbf{i}^j \\
 =& \begin{cases}
   -(x-1)^2 & \text{ if }j=0;\\
    (\mathbf{i}x^2+4x-\mathbf{i}) & \text{ if }j=1;\\
    (x+1)^2 & \text{ if }j=2;\\
    (-\mathbf{i}x^2+4x+\mathbf{i}) & \text{ if }j=3.\\
 \end{cases}\end{split}\]
 On the other hand, a computation of the determinant yields
    \[
    \begin{split} f_{X,\alpha}(T)=& \det (M(x)) \\
    =& x^{-4}\begin{pmatrix}
        3x & -1 & -x & -x^2\\
        -x^2 & 3x &  -1& -x\\
        -x & -x^2 & 3x & -1\\
         -1 & -x  & -x^2 & 3x 
    \end{pmatrix}\\
    =& -x^{-4}(x+1)^2(x-1)^2(x^4+14x^2+1)\\
    =& x^{-4}(x+1)^2(x-1)^2(ix^2+4x-i)(-ix^2+4x+i)\\
    = & x^{-4}\prod_{j=0}^3 P_j(T).
    \end{split}
    \]
This illustrates the Theorem \ref{determinant calculation}.
    
    Set $K_\psi=\Q_3(\mathbf{i})$, and $\cO$ (resp. $\varpi$) be the valuation ring (resp. unformizer) of $K_\psi$. We have that $\varpi=(3)$. By Theorem \ref{sum of Iwasawa invariants theorem}, we have that 
\begin{equation}\label{boring equation sum of iwasawa invariants}\begin{split}
        &\mu_3(X, \alpha)=\left(\sum_{j=0}^3 \mu_{\psi^j}\right);\\
        &\lambda_3(X, \alpha)=\left(\sum_{j=0}^3 \lambda_{\psi^j}\right)-1.
    \end{split}\end{equation}
For $j\in \{1, 2, 3\}$, then Lemma \ref{basic lemma on mu and lambda vanishing} asserts that the following are equivalent
\begin{enumerate}
    \item $\mu_{\psi^j}=0$ and $\lambda_{\psi^j}=0$,
    \item $\mathbf{i}^{j}+\mathbf{i}^{2j}+\mathbf{i}^{3j}\not \equiv 0\pmod{3}$.
\end{enumerate}
Since $\mathbf{i}^{j}+\mathbf{i}^{2j}+\mathbf{i}^{3j}=-1$, it follows that for $j\in \{1,2,3\}$, \[\mu_{\psi^j}=0\text{ and }\lambda_{\psi^j}=0.\]
We have that $P_0(T)=-T^2$, hence, $\mu_1=0$ and $\lambda_1=2$. Thus, from \eqref{boring equation sum of iwasawa invariants}, we have that 
\[\mu_3(X, \alpha)=0\text{ and }\lambda_3(X, \alpha)=1.\]
    
    One can visualize the $\Z_3$-tower as follows:
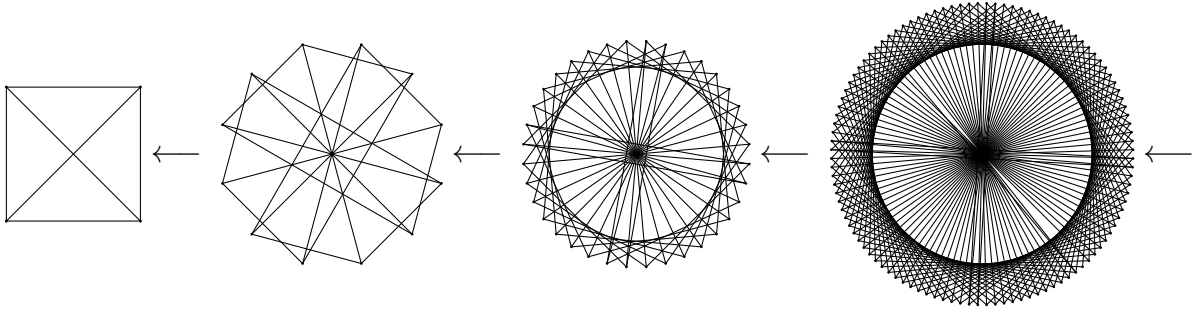
\begin{equation*}\label{picture}
\begin{tikzpicture}[baseline={([yshift=-0.6ex, xshift=-0.9ex] current bounding box.center)}, scale=0.8]
    \node[draw=none, minimum size=2.5cm, regular polygon, regular polygon sides=4] (a) {};
    \foreach \x in {1,2,3,4}
        \fill (a.corner \x) circle[radius=0.7pt];
    \path (a.corner 1) edge (a.corner 2);
    \path (a.corner 2) edge (a.corner 3);
    \path (a.corner 3) edge (a.corner 4);
    \path (a.corner 4) edge (a.corner 1);
     \path (a.corner 1) edge (a.corner 3);
      \path (a.corner 2) edge (a.corner 4);
\end{tikzpicture}
 \longleftarrow \, \,
\begin{tikzpicture}[baseline={([yshift=-0.6ex] current bounding box.center)}]
\node[draw=none,minimum size=3cm,regular polygon,regular polygon sides=12] (a) {};

\foreach \x in {1,2,...,12}
  \fill (a.corner \x) circle[radius=0.5pt];
  
\foreach \y\z in {1/6,1/7,2/4,2/8,3/5,3/9,4/9,4/10,5/7,5/11,6/8,6/12,7/12,8/10,9/11,10/3, 11/1, 12/2}
  \path (a.corner \y) edge (a.corner \z);
\end{tikzpicture}
 \longleftarrow \, \,
\begin{tikzpicture}[baseline={([yshift=-0.6ex] current bounding box.center)}]
\node[draw=none,minimum size=3cm,regular polygon,regular polygon sides=36] (a) {};

\foreach \x in {1,2,...,36}
  \fill (a.corner \x) circle[radius=0.5pt];
  
\foreach \y\z in {1/18, 1/19,2/10,2/20,3/11,3/21,4/12,4/22,5/13,5/23,6/14,6/24,7/15,7/25,8/16,8/26,9/17,9/27,10/27,10/28,11/19,11/29,12/20,12/30,13/21,13/31,14/22,14/32,15/23,15/33,16/24,16/34,17/25,17/35,18/26,18/36,19/36,20/28,21/29,22/30,23/31,24/32,25/33,26/34,27/35,28/9,29/1,30/2,31/3,32/4,33/5,34/6,35/7,36/8}
  \path (a.corner \y) edge (a.corner \z);
  
\end{tikzpicture}
 \longleftarrow \, \,
\begin{tikzpicture}[baseline={([yshift=-0.6ex] current bounding box.center)}]
\node[draw=none,minimum size=4cm,regular polygon,regular polygon sides=108] (a) {};

\foreach \x in {1,2,...,108}
  \fill (a.corner \x) circle[radius=0.5pt];
  
\foreach \y\z in {1/54,2/28,3/29,4/30,5/31,6/32,7/33,8/34,9/35,10/36,11/37,12/38,13/39,14/40,15/41,16/42,17/43,18/44,19/45,20/46,21/47,22/48,23/49,24/50, 25/51,26/52, 27/53}
  \path (a.corner \y) edge (a.corner \z);

\foreach \y\z in {1/55,2/56,3/57,4/58,5/59,6/60,7/61,8/62,9/63,10/64,11/65,12/66,13/67,14/67,15/69,16/70,17/71,18/72,19/73,20/74,21/75,22/76,23/77,24/78,25/79,26/80,27/81}
  \path (a.corner \y) edge  (a.corner \z);

  \foreach \y\z in {28/81,29/55,30/56,31/57,32/58,33/59,34/60,35/61,36/62,37/63,38/64,39/65,40/66,41/67,42/68,43/69,44/70,45/71,46/72,47/73,48/74,49/75,50/76,51/77,52/78,53/79,54/80}
  \path (a.corner \y) edge (a.corner \z);

 \foreach \y\z in {28/82,29/83,30/84,31/85,32/86,33/87,34/88,35/89,36/90,37/91,38/92,39/93,40/94,41/95,42/96,43/97,44/98,45/99,46/100,47/101,48/102,49/103,50/104,51/105,52/106,53/107,54/108}
  \path (a.corner \y) edge (a.corner \z);
 
  \foreach \y\z in {55/108,56/82,57/83,58/84,59/85,60/86,61/87,62/88,63/89,64/90,65/91,66/92,67/93,68/94,69/95,70/96,71/97,72/98,73/99, 74/100,75/101,76/102,77/103,78/104,79/105,80/106,81/107}
  \path (a.corner \y) edge (a.corner \z);

  \foreach \y\z in {82/27, 83/1, 84/2, 85/3, 86/4, 87/5, 88/6, 89/7,90/8, 91/9, 92/10, 93/11, 94/12, 95/13,96/14,97/15,98/16,99/17,100/18,101/19,102/20,103/21,104/22,105/23,106/24,107/25,108/26}
  \path (a.corner \y) edge (a.corner \z);
  \end{tikzpicture}
  \longleftarrow \, \,
\end{equation*}

\subsection{Example 2}
We consider another example where $\mu(X,\alpha)>0$. Let $G=\Z/6\Z:=\{0,1,2,3,4,5\}$, $S=\{1,2,3,4,5\}$, $X=\rm{Cay}(G, S)$ and $\ell=2$. Clearly, the condition $S=S^{-1}$ and the Assumption \ref{no vertices with val 1} are satisfied. Now, define the function $\beta: S\rightarrow \Z_2$ as follows:
     \[
    \beta(s)=\begin{cases}
          1, & \mathrm{if}\  s=1\\
          1, & \mathrm{if}\  s=2\\
          0, & \mathrm{if}\  s=3\\
          -1, & \mathrm{if}\  s=4\\
          -1, & \mathrm{if}\  s=5;
       \end{cases}
    \]
and $\alpha: E_X^+\rightarrow \Z_2$ the associated function. We check that the conditons (1)--(4) of Definition \ref{def of beta} are all satisfied.
\begin{enumerate}
    \item The image of $\beta$ is $\{-1, 0, 1\}$ and clearly generates $\Z_2$.
    \item The condition $\beta(-s)=-\beta(s)$ is easy to verify. 
    \item By definition, the image of $\beta$ lies in $\Z$.
    \item Consider the tuple $(1,1,1)\in S^3$. Then $\beta(1+1+1)=\beta(3)=0$,  $\beta(1)+\beta(1)+\beta(1)=3 $ and hence $\beta(1+1+1)\not\equiv 3\beta(1) \pmod{2}$.
\end{enumerate}

Now, the matrix $M(1+T)$ is given by
  \[
      \begin{pmatrix}
        5 & -(1+T)^{-1}& -(1+T)^{-1}& -1 & -(1+T) & -(1+T)\\
        -(1+T) & 5 & -(1+T)^{-1} & -(1+T)^{-1} & -1  &-(1+T)\\
        -(1+T) & -(1+T) & 5 & -(1+T)^{-1} & -(1+T)^{-1} & -1 \\
         -1 & -(1+T)& -(1+T) & 5  & -(1+T)^{-1}& -(1+T)^{-1}\\
         -(1+T)^{-1} & -1 & -(1+T) &-(1+T) & 5 & -(1+T)^{-1}\\
         -(1+T)^{-1} & -(1+T)^{-1} & -1 & -(1+T) & -(1+T) & 5
    \end{pmatrix}
    \]
Let $\psi\in \widehat{G}$ be the character which is defined by $\psi(n):=\op{exp}\left(\frac{2\pi \mathbf{i} 
 n}{6}\right)$. Let $\omega:=\op{exp}\left(\frac{2\pi \mathbf{i}}{6}\right)$. Therefore, $\psi(n)=\omega^n$. Note that  $\omega^3=-1$, $\omega^2+\omega=\mathbf{i}\sqrt{3}$ and $\omega^2-\omega=-1$, where $\mathbf{i}$ is the square root of -1.
 
 Set $P_j(T):=P_{\psi^j}(T)$ for $j=0, \dots, 5$; note that $m_\beta=1$. Setting $x:=(1+T)$, we find that 
 \[\begin{split}P_j(T)=& 5x-(\omega^j+\omega^{2j}+\omega^{3j}x+\omega^{4j}x^2+\omega^{5j}x^2) \\
 =& \begin{cases}
   -2(x-1)^2 & \text{ if }j=0;\\
    6x-\mathbf{i}\sqrt{3} (1-x^2)& \text{ if }j=1;\\
    4x+x^2+1 & \text{ if }j=2;\\
    6x & \text{ if }j=3.\\
     4x+x^2+1 & \text{ if }j=4;\\
     6x+\mathbf{i}\sqrt{3} (1-x^2)& \text{ if }j=5;
 \end{cases}\end{split}\]
Now a computation of the determinant gives us
    \[
    \begin{split} f_{X,\alpha}(T)=& \det (M(x)) \\
   &\begin{pmatrix}
5 & x^{-1} & x^{-1} & -1 & x & x\\
x & 5 & x^{-1} & x^{-1} & -1 & x\\
x & x & 5 & x^{-1} & x^{-1} & -1 \\
-1 & x & x & 5 & x^{-1} & x^{-1}\\
x^{-1} & -1 & x & x & 5 & x^{-1}\\
x^{-1} & x^{-1} & -1 & x & x & 5
\end{pmatrix}\\
    =& -36 x^{-5}(x-1)^2(x^2+4x+1)^2(x^4+10x^2+1)\\
= & x^{-6}\prod_{j=0}^5 P_j(T).
    \end{split}
    \]
This again illustrates the Theorem \ref{determinant calculation}.
    
    Set $K_\psi=\Q_2(\omega)$, and $\cO$ (resp. $\pi$) be the valuation ring (resp. unformizer) of $K_\psi$. We have that $\pi=(2)$. Again, by Theorem \ref{sum of Iwasawa invariants theorem}, we have that 
\begin{equation}\label{boring equation sum of iwasawa invariants2}\begin{split}
        &\mu_2(X, \alpha)=\left(\sum_{j=0}^5 \mu_{\psi^j}\right);\\
        &\lambda_2(X, \alpha)=\left(\sum_{j=0}^5 \lambda_{\psi^j}\right)-1.
    \end{split}\end{equation}
Now, we calculate each of these $\mu_{\psi^j}$'s and $\lambda_{\psi^j}$'s.
\begin{enumerate}
    \item For $P_0=-2T^2$, $\mu_{\psi^0}=1$, $\lambda_{\psi^0}=2$.
    \item For $P_1=\mathbf{i}\sqrt{3}(T^2+2T(1-\sqrt{3})-2\mathbf{i}\sqrt{3})$, $\mu_{\psi}=0$, $\lambda_{\psi}=2$.
    \item For $P_2=T^2+6T+6$, $\mu_{\psi^2}=0$, $\lambda_{\psi^2}=2$.
    \item For $P_3=6(T+1)$, $\mu_{\psi^3}=1$, $\lambda_{\psi^3}=0$.
    \item For $P_4=T^2+6T+6$, $\mu_{\psi^4}=0$, $\lambda_{\psi^4}=2$.
    \item For $P_5=-\mathbf{i}\sqrt{3}(T^2+2T(1-\sqrt{3})-2\mathbf{i}\sqrt{3})$, $\mu_{\psi^5}=0$, $\lambda_{\psi^5}=2$.
\end{enumerate}
Now, using equation \eqref{boring equation sum of iwasawa invariants2}, we get
\[\mu_2(X, \alpha)=2\text{ and }\lambda_2(X, \alpha)=9.\]
One can visualize the $\Z_2$-tower as follows:
\begin{equation*}\label{picture2}
\begin{tikzpicture}[baseline={([yshift=-0.6ex, xshift=-0.9ex] current bounding box.center)}, scale=0.5]
    \foreach \i in {1, 2, ..., 6} {
        \node[draw, circle, fill=black, inner sep=0.5pt] (V\i) at ({360/6 * (\i - 1)}:3) {};
    }
    \foreach \i in {1, 2, ..., 6} {
        \foreach \j in {\i,..., 6} {
            \ifnum\i<\j
                \draw (V\i) -- (V\j);
            \fi
        }
    }
    \end{tikzpicture}
    \longleftarrow \, \,
    \begin{tikzpicture}[baseline={([yshift=-0.6ex, xshift=-0.9ex] current bounding box.center)}, scale=0.5]
    \foreach \i in {1, ..., 12} {
        \node[draw, circle, fill=black, inner sep=0.5pt] (V\i) at ({360/12 * (\i - 1)}:3) {};
    }
    
    \foreach \source/\dest in {
        1/4, 2/3, 3/6, 4/5, 5/8, 6/7, 7/10, 8/9, 9/12, 10/11, 11/2, 12/1,
        1/6, 1/8, 2/5, 2/7, 1/10, 1/12, 2/9, 2/11, 3/8, 4/7, 3/10, 4/9,
        3/12, 4/11, 5/10, 6/9, 5/12, 6/11, 7/12, 8/11} 
        {
        \draw (V\source) -- (V\dest);
       }
\end{tikzpicture}
 \longleftarrow \, \,
\begin{tikzpicture}[baseline={([yshift=-0.6ex, xshift=-0.9ex] current bounding box.center)}, scale=0.33]

    \foreach \i in {1, ..., 24} {
        \node[draw, circle, fill=black, inner sep=0.5pt] (V\i) at ({360/24 * (\i - 1)}:5) {};
    }
    
    \foreach \i in {1, ..., 24} {
        \pgfmathsetmacro{\dest}{mod(\i + 1 - 1, 24) + 1} 
        \draw (V\i) -- (V\dest);
        
        \pgfmathsetmacro{\dest}{mod(\i + 2 - 1, 24) + 1} 
        \draw (V\i) -- (V\dest);
        
        \pgfmathsetmacro{\dest}{mod(\i + 3 - 1, 24) + 1} 
        \draw (V\i) -- (V\dest);
        
        \pgfmathsetmacro{\dest}{mod(\i + 4 - 1, 24) + 1} 
        \draw (V\i) -- (V\dest);
        
        \pgfmathsetmacro{\dest}{mod(\i + 5 - 1, 24) + 1} 
        \draw (V\i) -- (V\dest);
    }
\end{tikzpicture}
\longleftarrow \, \,
 \end{equation*}

\bibliographystyle{alpha}
\bibliography{references}

\begin{thebibliography}{HMSV24}

\bibitem[CP18]{divisorsandsandpiles}
Scott Corry and David Perkinson.
\newblock {\em Divisors and sandpiles}.
\newblock American Mathematical Society, Providence, RI, 2018.
\newblock An introduction to chip-firing.

\bibitem[DLRV24]{DLRV}
C\'{e}dric Dion, Antonio Lei, Anwesh Ray, and Daniel Valli\`eres.
\newblock On the distribution of {I}wasawa invariants associated to multigraphs.
\newblock {\em Nagoya Math. J.}, 253:48--90, 2024.

\bibitem[DV23]{DuBose/Vallieres:2022}
Sage DuBose and Daniel Valli\`eres.
\newblock On {$\Bbb Z^d_\ell$}-towers of graphs.
\newblock {\em Algebr. Comb.}, 6(5):1331--1346, 2023.

\bibitem[FW79]{FerreroWash}
Bruce Ferrero and Lawrence~C. Washington.
\newblock The {I}wasawa invariant {$\mu _{p}$} vanishes for abelian number fields.
\newblock {\em Ann. of Math. (2)}, 109(2):377--395, 1979.

\bibitem[Gon21]{Gonet:2021a}
Sophia~R. Gonet.
\newblock {\em Jacobians of finite and infinite voltage covers of graphs}.
\newblock ProQuest LLC, Ann Arbor, MI, 2021.
\newblock Thesis (Ph.D.)--The University of Vermont and State Agricultural College.

\bibitem[Gon22]{Gonet:2022}
Sophia~R. Gonet.
\newblock Iwasawa theory of {J}acobians of graphs.
\newblock {\em Algebr. Comb.}, 5(5):827--848, 2022.

\bibitem[HMSV24]{HammerMattmanSandsVallieres}
Kyle Hammer, Thomas~W. Mattman, Jonathan~W. Sands, and Daniel Valli\`eres.
\newblock The special value {$u=1$} of {A}rtin-{I}hara {$L$}-functions.
\newblock {\em Proc. Amer. Math. Soc.}, 152(2):501--514, 2024.

\bibitem[Iwa59]{IwasawaMain}
Kenkichi Iwasawa.
\newblock On {$\Gamma $}-extensions of algebraic number fields.
\newblock {\em Bull. Amer. Math. Soc.}, 65:183--226, 1959.

\bibitem[KM22]{Kleine/Muller:2022}
S\"{o}ren Kleine and Katharina M\"{u}ller.
\newblock On the growth of the {J}acobian in $\mathbb{Z}_{p}^{l}$-voltage covers of graphs.
\newblock {\em {P}reprint, arxiv:2211.09763}, 2022.

\bibitem[LM24]{LeiMuller}
Antonio Lei and Katharina M\"{u}ller.
\newblock On the zeta functions of supersingular isogeny graphs and modular curves.
\newblock {\em Arch. Math. (Basel)}, 122(3):285--294, 2024.

\bibitem[MV23]{mcgownvallieresII}
Kevin McGown and Daniel Valli\`eres.
\newblock On abelian {$\ell$}-towers of multigraphs {II}.
\newblock {\em Ann. Math. Qu\'{e}.}, 47(2):461--473, 2023.

\bibitem[MV24]{mcgownvallieresIII}
Kevin McGown and Daniel Valli\`eres.
\newblock On abelian {$\ell$}-towers of multigraphs {III}.
\newblock {\em Ann. Math. Qu\'{e}.}, 48(1):1--19, 2024.

\bibitem[Nor98]{Northshield}
Sam Northshield.
\newblock A note on the zeta function of a graph.
\newblock {\em J. Combin. Theory Ser. B}, 74(2):408--410, 1998.

\bibitem[RV22]{Ray/Vallieres:2022}
Anwesh Ray and Daniel Valli{\`e}res.
\newblock An analogue of {K}ida's formula in graph theory.
\newblock {\em {P}reprint, arXiv:2209.04890}, 2022.

\bibitem[Ter11]{Terras:2011}
Audrey Terras.
\newblock {\em Zeta functions of graphs}, volume 128 of {\em Cambridge Studies in Advanced Mathematics}.
\newblock Cambridge University Press, Cambridge, 2011.
\newblock A stroll through the garden.

\bibitem[Val21]{Vallieres:2021}
Daniel Valli\`eres.
\newblock On abelian {$\ell$}-towers of multigraphs.
\newblock {\em Ann. Math. Qu\'{e}.}, 45(2):433--452, 2021.

\end{thebibliography}
\end{document}